\documentclass[ECP,preprint]{ejpecp} 

\usepackage{xcolor}
\usepackage{mathrsfs}
\usepackage[inline, shortlabels]{enumitem}
 \usepackage{relsize}
 \usepackage{nicefrac}

\SHORTTITLE{Ill-posedness of the pure-noise Dean--Kawasaki equation}

\TITLE{Ill-posedness of the pure-noise Dean--Kawasaki equation\support{This research was partly funded by the Austrian Science Fund (FWF) project \href{https://doi.org/10.55776/ESP208}{{10.55776/ESP208}.}
}\
    \thanks{
    This work started while the first named author was a guest of the Mathematics Department of the University of Hamburg.
    It is his pleasure to thank the Department for their kind hospitality.
    The authors are grateful to Max-Konstantin von Renesse for some helpful comments.}
    } 

\AUTHORS{%
  Lorenzo Dello Schiavo\footnote{Dipartimento di Matematica, Università degli Studi di Roma ``Tor Vergata'', Via della Ricerca Scientifica, 1, 00133 Roma, Italy.
    \EMAIL{delloschiavo@mat.uniroma2.it} .
    }
  \and 
Vitalii Konarovskyi\footnote{Faculty of Mathematics informatics and natural sciences, University of Hamburg, Bundesstraße 55, 20146 Hamburg, Germany. \EMAIL{vitalii.konarovskyi@uni-hamburg.de}.}
\footnote{Institute of Mathematics of NAS of Ukraine, Tereschenkivska st. 3, 01024 Kyiv, Ukraine.}
}


\AMSSUBJ{60H15; 60G57; 82C31} 

\SUBMITTED{} 
\ACCEPTED{} 

\VOLUME{0}
\YEAR{0}
\PAPERNUM{0}
\DOI{0}

\ABSTRACT{We prove that the Dean--Kawasaki-type stochastic partial differential equation
\[
\partial \rho= \nabla\cdot (\sqrt{\rho\,}\, \xi) + \nabla\cdot \tparen{\rho\, H(\rho)} \comma
\]
with vector-valued space-time white noise~$\xi$, does not admit solutions for any initial measure and any vector-valued bounded measurable function~$H$ on the space of measures.
This applies in particular to the pure-noise Dean--Kawasaki equation ($H\equiv 0$).
The result is sharp, in the sense that solutions are known to exist for some unbounded~$H$.
}

\newtheorem*{theorem*}{Theorem}

\usepackage{xparse}

\ExplSyntaxOn
\NewDocumentCommand{\makeabbrev}{mmm}
 {
  \yoruk_makeabbrev:nnn { #1 } { #2 } { #3 }
 }

\cs_new_protected:Npn \yoruk_makeabbrev:nnn #1 #2 #3
 {
  \clist_map_inline:nn { #3 }
   {
    \cs_new_protected:cpn { #2 } { #1 { ##1 } }
   }
 }
 \ExplSyntaxOff

\makeabbrev{\textbf}{tbf#1}{a,b,c,d,e,f,g,h,i,j,k,l,m,n,o,p,q,r,s,t,u,v,w,x,y,z,A,B,C,D,E,F,G,H,I,J,K,L,M,N,O,P,Q,R,S,T,U,V,W,X,Y,Z}

\makeabbrev{\textbf}{bf#1}{a,b,c,d,e,f,g,h,i,j,k,l,m,n,o,p,q,r,s,t,u,v,w,x,y,z,A,B,C,D,E,F,G,H,I,J,K,L,M,N,O,P,Q,R,S,T,U,V,W,X,Y,Z}

\makeabbrev{\textsf}{tsf#1}{a,b,c,d,e,f,g,h,i,j,k,l,m,n,o,p,q,r,s,t,u,v,w,x,y,z,A,B,C,D,E,F,G,H,I,J,K,L,M,N,O,P,Q,R,S,T,U,V,W,X,Y,Z}

\makeabbrev{\mathsf}{mss#1}{a,b,c,d,e,f,g,h,i,j,k,l,m,n,o,p,q,r,s,t,u,v,w,x,y,z,A,B,C,D,E,F,G,H,I,J,K,L,M,N,O,P,Q,R,S,T,U,V,W,X,Y,Z}

\makeabbrev{\mathfrak}{mf#1}{a,b,c,d,e,f,g,h,i,j,k,l,m,n,o,p,q,r,s,t,u,v,w,x,y,z,A,B,C,D,E,F,G,H,I,J,K,L,M,N,O,P,Q,R,S,T,U,V,W,X,Y,Z,
sl,gl
}

\makeabbrev{\mathrm}{mrm#1}{a,b,c,d,e,f,g,h,i,j,k,l,m,n,o,p,q,r,s,t,u,v,w,x,y,z,A,B,C,D,E,F,G,H,I,J,K,L,M,N,O,P,Q,R,S,T,U,V,W,X,Y,Z}

\makeabbrev{\mathbf}{mbf#1}{a,b,c,d,e,f,g,h,i,j,k,l,m,n,o,p,q,r,s,t,u,v,w,x,y,z,A,B,C,D,E,F,G,H,I,J,K,L,M,N,O,P,Q,R,S,T,U,V,W,X,Y,Z}

\makeabbrev{\mathcal}{mc#1}{A,B,C,D,E,F,G,H,I,J,K,L,M,N,O,P,Q,R,S,T,U,V,W,X,Y,Z}

\makeabbrev{\mathbb}{mbb#1}{A,B,C,D,E,F,G,H,I,J,K,L,M,N,O,P,Q,R,S,T,U,V,W,X,Y,Z}

\makeabbrev{\mathscr}{ms#1}{A,B,C,D,E,F,G,H,I,J,K,L,M,N,O,P,Q,R,S,T,U,V,W,X,Y,Z}

\makeabbrev{\mathrm}{#1}{
id,ran,rk,diag,stab,ann,conv,pr,ev,tr,End,Hom,sgn,im,op,can,fin,ext,red,tot,lex,Aut,Inn,unit,
rot,usc,lsc,Lip,lip,bSymLip,osc,AC,loc,coz,z,
supp,Opt,Adm,Cpl,Geo,GeoOpt,GeoAdm,GeoCpl,reg,res,
bd,co,Ric,Exp,dExp,dist,seg,Seg,cut,fcut,Cut,SDiff,Iso,Isom,diam,cl,Homeo,Diff,Der,vol,dvol,inj,relint, Graph, sub,
var,law,Var,Poi,Gam,pa,so,iso,fs,inv,pqi,mix,erg,form,
TestF,
ob,cod,inp,
}

\makeabbrev{\mathsf}{#1}{CD,BE,MCP,Ent,wMTW,MTW,Ch,RCD,EVI,Rad,dRad,SL,cSL,dSL,ScL,Irr,SC,wFe,VA,MetMeas,UMeas,CSMet,Met,USp,Meas,Mbl,alg,Alg}

\makeabbrev{\mathsc}{#1}{mmaf,cg}

\newcommand{\emp}{\varnothing}
\newcommand{\R}{\mathbb{R}}
\newcommand{\N}{\mathbb{N}}
\newcommand{\Z}{\mathbb{Z}}

\newcommand{\fstop}{\ \text{.}}
\newcommand{\comma}{\ \text{,} \ }

\newcommand{\eps}{\varepsilon}

\newcommand{\Leb}{\msL}

\newcommand{\Mbp}{\msM_b^+}

\renewcommand{\complement}{\mathrm{c}}

\newcommand{\mathsc}[1]{\text{\textsc{#1}}}
\newcommand{\emparg}{{\,\cdot\,}}

\DeclareMathOperator{\eqdef}{\coloneqq}

\newcommand{\car}{\mathbf{1}}

\newcommand{\diff}{\mathop{}\!\mathrm{d}}						

\newcommand{\abs}[1]{\left\lvert#1\right\rvert}						
\newcommand{\norm}[1]{\left\lVert#1\right\rVert}					
\newcommand{\set}[1]{\left\{#1\right\}}							
\newcommand{\paren}[1]{\left(#1\right)}							
\newcommand{\tparen}[1]{\big({#1}\big)}

\newcommand{\quadvar}[1]{\left[#1\right]}							

\newcommand{\restr}[1]{\big\rvert_{#1}}

\newcommand{\seq}[1]{\paren{#1}}								
\newcommand{\tseq}[1]{{\big(#1\big)}}

\newcommand{\Cb}{\mcC_b}									

\allowdisplaybreaks

\begin{document}


\section{Introduction and the main result}
Let~$\mbbM^d$ be either the standard $d$-dimensional Euclidean space~$\R^d$ or the flat $d$-dimension\-al torus~$\mbbT^d$, $d\geq 1$.
For~$k\in \N_0$, we let~$\Cb^k$ be the space of all continuous and bounded real-valued functions on~$\mbbM^d$ with continuous and bounded derivatives up to order~$k$, and we set~$\Cb\eqdef\Cb^0$, endowed with the uniform norm~$\norm{\emparg}_0$.
For a Borel measure~$\mu$ on~$\mbbM^d$ and a Borel function~$f\colon \mbbM^d\to\R$, we write~$\mu f\eqdef \int f \diff\mu$ whenever the integral makes sense.
We denote by~$\Mbp$ the space of all positive finite Borel measures on~$\mbbM^d$, endowed with the narrow topology, i.e.\ the coarsest topology for which all the functionals~$\mu\mapsto \mu f$, with~$f\in\Cb$, are continuous.

On~$\mbbM^d$ we consider the \emph{Dean--Kawasaki equation}
\begin{equation}\label{eq:DK}
\diff\mu_t =  \alpha\, \Delta\mu_t \diff t+ G(\mu_t)\diff t + \nabla \cdot (\sqrt{\mu_t}\ \xi)\comma
\end{equation}
where~$\alpha\geq 0$ is a parameter,~$G\colon \Mbp\to\R$ is Borel measurable,~$\xi$ is an $\R^d$-valued space-time white noise, and~$\seq{\mu_t}_{t\geq 0}$ is an $\Mbp$-valued stochastic process with a.s.\ continuous paths.

\medskip

The equation with~$\alpha>0$ has been proposed by K.~Kawasaki in~\cite{Kaw94} and, independently, by D.~S.~Dean in~\cite{Dea96}, to describe the density function of a system of $N\gg 1$ particles subject to a diffusive Langevin dynamics, with the noise~$\xi$ describing the particles' thermal fluctuations.
Equations like~\eqref{eq:DK} ---~possibly with a different non-linearity in the noise term~--- fall within the class of Ginzburg--Landau stochastic phase field models, and effectively describe super-cooled liquids, colloidal suspensions, the glass-liquid transition, bacterial patterns, and other systems; see, e.g., the recent review~\cite{VruLoeWitt20}.

From a mathematical point of view, these equations model in the continuum the \emph{fluctuating hydrodynamic theory} of interacting particle systems; see, e.g.,~\cite{QuaRezVar99,DirStaZim16,BenKipLan95,FehGes23} and the review~\cite{BerDeSGabJon15}.
A specific interest in the case of~\eqref{eq:DK} ---~i.e., with a square-root non-linearity in the noise term~--- is partially motivated by the structure of the noise in connection with the geometry of the $L^2$-Kantorovich--Rubinstein--Wasserstein space~$(\msP_2,W_2)$.
Indeed, in the free case ($H\equiv 0$), a solution~$\mu_t$ to~\eqref{eq:DK} with~$\alpha=1$ is an intrinsic random perturbation of the gradient flow of the Boltzmann--Shannon entropy on~$\msP_2$ by a noise~$\xi$ distributed according to the energy dissipated by the system, i.e.\ by the natural isotropic noise arising from the Riemannian structure of~$\msP_2$, see~\cite{JacZim14,KonvRe17,KonvRe18,AdaDirPelZim11}.

We would like to stress that we consider here Dean--Kawasaki-type equations with \emph{white} noise: a very fruitful theory has been developed for similar equations with \emph{colored}, \emph{truncated}, or otherwise \emph{approximated} noise (both of It\^o and Stratonovich type), abstractly~\cite{CorFis23,FehGes23,FehGes24b,FehGes24,GesGvaKon22}, numerically~\cite{CorFis23b,CorFisIngRai23,DjuKrePer24}, and ---~for both colored and white noise~--- approaching concrete applications~\cite{DirStaZim16,EmbDirZimRei18}.

\subsection{Main result}
A rigorous definition of solutions to~\eqref{eq:DK} was introduced by V.~Konarovskyi, T.~Leh\-mann, and M.-K.~von~Renesse in~\cite{KonLehvRe19} for~$G\equiv 0$, and in~\cite{KonLehvRe19b} when
\begin{equation}\label{eq:DefH}
G(\mu) = \nabla \cdot \tparen{\mu\, H(\mu)}
\end{equation}
for~$H\colon \Mbp\to\R^d$, as we now recall.

\begin{definition}[Martingale solutions, cf.~{\cite[Dfn.~1]{KonLehvRe19b}}]\label{d:Solution}
Fix~$T\in (0,\infty)$ and let~$(\Omega,\msF,\mbfP)$ be a complete probability space. 
A continuous~$\Mbp$-valued process~$\mu_\bullet\eqdef\seq{\mu_t}_{t\in[0,T]}$ on~$(\Omega,\msF)$ is a \emph{solution to~\eqref{eq:DK}} (\emph{up to time~$T$}) if, for each~$f\in\Cb^2$ the process~$M^f_\bullet\eqdef\tseq{M^f_t}_{t\in [0,T]}$ with
\[
M^f_t\eqdef \mu_t f - \mu_0 f- \int_0^t \mu_s\tparen{ \tfrac{\alpha}{2}\Delta f+ \nabla f \cdot H(\mu_s)}\diff s\comma \qquad t\in [0,T] \comma
\]
is a continuous $\mbfP$-martingale on~$(\Omega,\msF)$ with respect to the filtration~$\msF_\bullet\eqdef \seq{\msF_t}_{t\in[0,T]}$ generated by~$\mu_\bullet$, with quadratic variation
\[
\quadvar{M^f}_t = \int_0^t \mu_s \abs{\nabla f}^2 \diff s\comma \qquad t \in [0,T]\fstop
\]
\end{definition}

In the case when~$\alpha>0$ and~$H(\mu)=\nabla \frac{\delta F(\mu)}{\delta\mu}$ for some sufficiently smooth and \emph{bounded}~$F\colon \Mbp\to\R$, Konarovskyi, Lehmann, and von~Renesse have shown in~\cite{KonLehvRe19,KonLehvRe19b} that~\eqref{eq:DK} admits  solutions if and only if the initial datum~$\mu_0$ is the empirical measure of a finite particle system, i.e.~$\mu_0$ is a purely atomic measure and each atom has mass~$1/\alpha$.
In this case, the solution~$\mu_\bullet$ exists for all times, is unique and identical with the empirical measure of the Langevin particle systems with mean-field interaction~$F$.
Further extensions of these rigidity results were subsequently obtained by Konarovskyi and M\"uller in~\cite{KonMue23} and by M\"uller, von Renesse, and Zimmer in~\cite{MuevReZim24}.

Their technique, however, does not apply to the case~$\alpha=0$, hence in particular it does not cover the pure-noise Dean--Kawasaki equation.
Here, we complete the picture by addressing precisely this case.

\begin{theorem}\label{t:Main}
Let~$\alpha=0$ and~$G(\mu)= \nabla \cdot \tparen{\mu\, H(\mu)}$ for some bounded Borel~$H\colon\Mbp\to\R^d$. Then~\eqref{eq:DK} has no solutions for any initial condition~$\mu_0\in\Mbp$.
\end{theorem}

This result is sharp, in the sense that existence of solutions was shown by Konarovskyi and von Renesse in~\cite{KonvRe18,KonvRe17} for the Dean--Kawasaki equation on the real line with singular drift
\begin{align}\label{eq:DKDrift}
\diff\mu_t =  \sum_{x: \mu_t\set{x}>0} \Delta\delta_x \diff t + \nabla \cdot (\sqrt{\mu_t}\ \xi)\comma
\end{align}
that is, in the case when~$H$ in~\eqref{eq:DefH} is \emph{un}bounded.
Existence of solutions to~\eqref{eq:DKDrift} were eventually constructed by the first named author also on compact manifolds~\cite{LzDS17+} and in other more general settings~\cite{LzDS24}.

\section{Proofs}
For any real-valued function~$f$ we denote by~$\Sigma_f$ the \emph{singular set} of~$f$, i.e.\ the set of points in the domain of~$f$ at which~$f$ is \emph{not} differentiable.
If not stated otherwise,~$(\Omega,\msF,\mbfP)$ is a complete probability space, and we denote by~$\mbfE$ the $\mbfP$-expectation.
Further let~$\mu_\bullet$ be a solution to~\eqref{eq:DK} up to time~$T$ on~$(\Omega,\msF,\mbfP)$ and assume that
\begin{equation}\label{eq:AssumptionRandom}
\mbfE[\mu_0\mbbM^d]<\infty \fstop
\end{equation}
(Note that~\eqref{eq:AssumptionRandom} is trivially satisfied if~$\mu_0$ is deterministic.)

We start with some preparatory lemmas.

\begin{lemma}\label{l:MassConservation}
If~$\mu_\bullet$ is a solution to~\eqref{eq:DK} up to time~$T$, then~$\mu_t\mbbM^d=\mu_0\mbbM^d$ a.s.\ for all~$t\in [0,T]$.

\begin{proof}
Choosing~$f=\car$ in the martingale problem in Definition~\ref{d:Solution}, we have~$\quadvar{M^\car}_t=0$ for all times.
It follows that~$\mu_t\mbbM^d=M^\car_t$ is a.s.\ a constant martingale, and therefore~$\mu_t\mbbM^d=\mu_0\mbbM^d$ a.s.\ for all times.
\end{proof}
\end{lemma}

For each~$t>0$ define a measure~$\mu^*_t$ on~$\mbbM^1$ as
\begin{equation}\label{eq:AuxMeas}
\mu^*_t \eqdef  \mbfE\int_0^t \int_{\mbbM^{d-1}}\mu_s(\emparg, \diff x_2,\dotsc, \diff x_d) \diff s
\end{equation}
Whenever the assumption in~\eqref{eq:AssumptionRandom} is satisfied, $\mu^*_t$ is a finite measure by Lemma~\ref{l:MassConservation}, hence the set~$A_t$ of its atoms is at most countable.

\medskip

Throughout the rest of this work we assume that~\eqref{eq:AssumptionRandom} holds, we fix~$T>0$ and we set~$A\eqdef A_T$.

\begin{lemma}\label{l:AuxFunction}
There exists a continuous function~$g\colon \mbbM^1\to\R$ with the following properties:
\begin{enumerate}[$(i)$]
\item\label{i:l:AuxFunction:1} $g$ is piecewise affine, non-negative, and bounded;
\item\label{i:l:AuxFunction:3} $\Sigma_g\cap A=\emp$;
\item\label{i:l:AuxFunction:2} $\Sigma_g$ is at most countable and $\abs{g'}=1$ on~$\Sigma_g^\complement$;
\end{enumerate}
\end{lemma}

\begin{proof}
We may dispense with showing that~$g$ is non-negative.
Indeed, suppose we have found some function~$g$ with all the required properties except non-negativity.
Then,~$g-\inf g$ still satisfies all these properties, since it has the same singular set as~$g$, and is additionally non-negative.

Assume~$\mbbM^1=\R$.
Fix~$y_0\in A^\complement$, and define inductively a countable set~$Y\eqdef\set{y_k}_{k\in\Z}$ in the following way:
if~$k\in \Z^\pm$, choose~$y_k\in A^\complement$ such that~$\abs{y_k-(y_{k\mp 1}\pm 1)} \leq 2^{-k}$.
Further set
\begin{equation*}
a_k\eqdef 
\begin{cases}
\displaystyle\sum_{i=1}^k (-1)^{i-1}(y_i-y_{i-1}) & \text{if } k\in \Z^+ \comma
\\
0 & \text{if } k=0\comma
\\
\displaystyle\sum_{i=k}^{-1} (-1)^{i+1} (y_{i+1}-y_i) & \text{if } k\in \Z^- \fstop
\end{cases}
\end{equation*}
In this way,~$Y\subset A^\complement$ and~$\abs{a_k} \leq 2$ for every~$k\in\Z$.
It follows that the linear spline~$g$ interpolating the points~$\seq{(y_k,a_k)}_{k\in\Z}$ has all the desired properties (with the possible exception of non-negativity) and in particular satisfies~$\norm{g}_0\leq \sup_k \abs{a_k}\leq  2$.

Assume~$\mbbM^1=\mbbT^1$. All sets and points in the rest of the proof are regarded~$\mod 1$.
Since~$A$ is countable,~$A_1\eqdef A\cup (A+\set{\nicefrac{1}{2}})$ is countable too, and we can choose~$y\not\in A_1$, which implies that~$y+\nicefrac{1}{2}\not\in A_1$ as well.
Then, the function~$g$ defined as the piecewise affine function with singular set~$\set{y,y+\nicefrac{1}{2}}$ and interpolating the points~$(y,0)$ and~$(y+\nicefrac{1}{2},\nicefrac{1}{2})$ has all the desired properties.
\end{proof}

\begin{proposition}\label{p:MartConv}
Fix~$T\in (0,\infty)$, and let~$\mu_\bullet\eqdef \seq{\mu_t}_{t\leq T}$ be a solution ---~if any exists~--- to~\eqref{eq:DK} for~$\alpha=0$ up to time~$T$.
Further suppose that:
\begin{itemize*}[{ }]
\item $f_n\colon \mbbM^d\to\R$ is a function in~$\Cb^2$ for each~$n\in\N$, 
\item $f\colon \mbbM^d\to\R$ is a function in~$\Cb^0$,
\item $h\colon \mbbM^d\to\R^d$ is a Borel measurable function with~$h\equiv \nabla f$ on~$\Sigma_f^\complement$,
\end{itemize*}
satisfying
\begin{enumerate}[$(a)$]
\item\label{i:p:MartConv:1} $\displaystyle\lim_n f_n=f$ uniformly on~$\mbbM^d$;
\item\label{i:p:MartConv:3} $\displaystyle \lim_n \int_0^T \mu_s \abs{\nabla f_n - h} \diff s=0$ a.s.
\end{enumerate}

Then, the process~$M_\bullet\eqdef \seq{M_t}_{t\in [0,T]}$ with
\begin{equation}\label{eq:pMartConv:0}
M_t\eqdef \mu_t f - \mu_0 f - \int_0^t \mu_s \tparen{h \cdot H(\mu_s)}\diff s\comma \qquad t\in [0,T] \comma
\end{equation}
is a martingale with respect to the filtration~$\msF_\bullet\eqdef \seq{\msF_t}_{t\in [0,T]}$ generated by~$\mu_\bullet$, with quadratic variation
\begin{equation}\label{eq:pMartConv:00}
\quadvar{M}_t = \int_0^t \mu_s \abs{h}^2 \diff s\comma \qquad t\in [0,T]\fstop
\end{equation}
\end{proposition}

\begin{proof}
By Definition~\ref{d:Solution}, for every~$n\in \N$, the processes~$M^n_\bullet\eqdef \seq{M^n_t}_{t\in [0,T]}$ with
\begin{equation}\label{eq:p:MartConv:3}
M^n_t\eqdef \mu_t f_n - \mu_0 f_n -\int_0^t \mu_s\tparen{\nabla f_n\cdot H(\mu_s)} \diff s \comma \qquad t\in [0,T]\comma
\end{equation}
is a continuous martingale w.r.t.\ the \emph{same} filtration~$\msF_\bullet$, with quadratic variation
\begin{equation}\label{eq:p:MartConv:4}
\quadvar{M^n}_t= \int_0^t \mu_s\abs{\nabla f_n}^2 \diff s \comma \qquad t\in [0,T] \fstop
\end{equation}

The conclusion will follow letting~$n\to\infty$ in~\eqref{eq:p:MartConv:3} and applying~\cite[Lem.~B.11]{CheEng05}, provided we show that~$M^n_\bullet$ converges to~$M_\bullet$ in probability uniformly on~$[0,T]$, that is
\begin{equation}\label{eq:p:MartConv:6}
\mbbP\text{-}\lim_n \sup_{t\leq T} \abs{M^n_t - M_t} = 0 \fstop
\end{equation}
We show the stronger statement that
\[
\lim_n \sup_{t\leq T} \abs{M^n_t - M_t} = 0 \quad \text{a.s.}
\]

Indeed, by~\ref{i:p:MartConv:1},
\begin{equation}\label{eq:p:MartConv:1}
\lim_n \abs{\mu_0f_n-\mu_0 f}=0 \quad \text{a.s.}
\end{equation}
By Lemma~\ref{l:MassConservation} and by~\ref{i:p:MartConv:1},
\begin{equation}\label{eq:p:MartConv:2}
\lim_n \sup_{t\in [0,T]} \abs{\mu_t f_n - \mu_t f} \leq \lim_n \sup_{t\in [0,T]} \mu_t\mbbM^d \norm{f_n-f}_0 = \mu_0 \mbbM^d \lim_n \norm{f_n-f}=0 \fstop
\end{equation}
By Cauchy--Schwarz inequality, uniform boundedness of~$H\colon \Mbp\to\R^d$, and~\ref{i:p:MartConv:3},
\begin{align}
\nonumber
\lim_n \sup_{t\in [0,T]}& \abs{\int_0^t \mu_s\tparen{\nabla f_n\cdot H(\mu_s)} \diff s- \int_0^t \mu_s\tparen{h\cdot H(\mu_s)} \diff s} \leq
\\
\nonumber
&\leq \lim_n \sup_{t\in [0,T]} \int_0^t \mu_s\abs{\paren{\nabla f_n-h}\cdot H(\mu_s)} \diff s
\\
\label{eq:p:MartConv:5}
&\leq\norm{H}_0 \lim_n \int_0^T \mu_s \abs{\nabla f_n-h} \diff s=0 \fstop
\end{align}

Combining~\eqref{eq:p:MartConv:1},~\eqref{eq:p:MartConv:2}, and~\eqref{eq:p:MartConv:5} shows~\eqref{eq:p:MartConv:6} and thus the assertion.
\end{proof}

We are now ready to prove our main result.

\begin{proof}[Proof of Theorem~\ref{t:Main}]
Fix~$\mu_0\in\Mbp$ and set~$c\eqdef \mu_0\mbbM^d>0$.
We argue by contradiction that there exists a solution~$\seq{\mu_t}_t$ to~\eqref{eq:DK} starting at~$\mu_0$.

Let~$g$ be the function constructed in Lemma~\ref{l:AuxFunction} and, for every~$\eps>0$, define~$g_\eps\colon \mbbM^1\to\R$ as a regularization of~$g$ satisfying:
\begin{enumerate*}[$(a_g)$]
\item\label{i:t:Main:1g} $g_\eps\in\Cb^2$ and $g_\eps$ converges to~$g$ uniformly on~$\mbbM^1$ as $\eps\downarrow 0$;
\item\label{i:t:Main:2g} $g_\eps'$ converges to~$g'\equiv 1$ locally uniformly away from~$\Sigma_g$ as $\eps\downarrow 0$;
\item\label{i:t:Main:3g} $\abs{g_\eps'}\leq 1$ everywhere on~$\mbbM^1$.
\end{enumerate*}
Finally, define~$f_\eps\colon \mbbM^d\to\R$ and~$f\colon \mbbM^d\to\R$ by~$f_\eps(x)\eqdef g_\eps(x_1)$ and~$f(x)\eqdef g(x_1)$ respectively, where~$x=\seq{x_1,\dotsc, x_d}\in \mbbM^d$.
Now, let~$\eps\eqdef 1/n$  and put, for simplicity of notation,~$f_n\eqdef f_{\eps_n}$
From~\ref{i:t:Main:1g}-\ref{i:t:Main:3g} above we deduce the analogous properties for~$f_n$ and~$f$, that is
\begin{enumerate}[$(a_f)$]
\item\label{i:t:Main:1} $f_n\in\Cb^2$ converges to~$f$ uniformly on~$\mbbM^d$ as $n\to\infty$;
\item\label{i:t:Main:2} $\nabla f_n$ converges to~$\nabla f$ locally uniformly away from~$\Sigma_f$ as $n\to\infty$;
\item\label{i:t:Main:3} $\abs{\nabla f_n}\leq 1$ everywhere on~$\mbbM^d$.
\end{enumerate}

\paragraph{Step 1} We start by verifying the assumptions in Proposition~\ref{p:MartConv}.
The singular set~$\Sigma_f$ of~$f$ satisfies~$\Sigma_f=\Sigma_g\times \mbbM^{d-1}$. Thus, for every~$t\in [0,T]$,
\begin{equation*}
\mbfE\int_0^t \mu_s \Sigma_f \diff s \leq \mbfE\int_0^T \mu_s\Sigma_f \diff s = \mbfE\int_0^T \mu_s(\Sigma_g\times \mbbM^{d-1}) \diff s = \mu^*_T\Sigma_g=0
\end{equation*}
by Lemma~\ref{l:AuxFunction}\ref{i:l:AuxFunction:3}, and therefore
\begin{equation}\label{eq:t:Main:1}
\int_0^t \mu_s \Sigma_f \diff s = 0 \quad \text{a.s.}\comma\qquad t\in [0,T] \fstop
\end{equation}
Respectively: by~\eqref{eq:t:Main:1}; since~$(\nabla f)(x)= g_\eps'(x_1)=1$ on~$\Sigma_f^\complement$ by definition of~$f$ and Lemma~\ref{l:AuxFunction}\ref{i:l:AuxFunction:2}; and by Lemma~\ref{l:MassConservation},
\begin{equation}\label{eq:t:Main:2}
\int_0^t \mu_s \abs{\nabla f}^2 \diff s = \int_0^t \mu_s \restr{\Sigma_f^\complement} \abs{\nabla f}^2 \diff s = \int_0^t \mu_s \car \diff s = c\, t \quad \text{a.s.}\comma\qquad t\in [0,T]\comma
\end{equation}
which shows in particular that the integral in the left-hand side of~\eqref{eq:t:Main:2} is well-defined for every~$t\in [0,T]$ and thus that
\begin{equation*}
\mu_s\abs{\nabla f}^2=\mu_s\car \quad \text{is a.s.\ well-defined for a.e.~} s\in [0,T] \fstop
\end{equation*}
This shows that in Proposition~\ref{p:MartConv} we may choose~$h=\nabla f$.

Fix~$s\in [0,T]$. Since~$\mu_s$ is a.s.\ a finite measure, by the convergence in~\ref{i:t:Main:2} and Dominated Convergence in~$L^1(\mu_s)$ 
because of~\ref{i:t:Main:3},
\begin{equation}\label{eq:t:Main:4}
\lim_{n\to\infty} \int \abs{\nabla f_n -\nabla f} \diff\mu_s =0\qquad \text{a.s.}\comma \quad \text{for a.e.~} s\in [0,T] \fstop
\end{equation}
By~\ref{i:t:Main:3} and Lemma~\ref{l:MassConservation},
\begin{equation*}
\mu_s \abs{\nabla f_n-\nabla f}\leq 2\mu_s \car = 2c\qquad  \text{a.s.} \comma \quad \text{for a.e.~} s\in [0,T] \comma \quad n\in \N\fstop
\end{equation*}
Thus, the function~$s\mapsto \mu_s \abs{\nabla f_\eps -\nabla f}$ is a.s.\ $\Leb^1$-essentially bounded on~$[0,T]$ uniformly in~$n$.
By the convergence in~\eqref{eq:t:Main:4} for a.e.~$s\in [0,T]$ and Dominated Convergence in~$L^1([0,T])$ with dominating function~$2c\in L^1([0,T])$,
\begin{equation}\label{eq:t:Main:5}
\lim_{n\to\infty} \int_0^T \mu_s\abs{\nabla f_n - \nabla f} \diff s =0  \qquad \text{a.s.}
\end{equation}

Note that~\ref{i:t:Main:1} verifies the assumption in Proposition~\ref{p:MartConv}\ref{i:p:MartConv:1}, while~\eqref{eq:t:Main:5} verifies Proposition~\ref{p:MartConv}\ref{i:p:MartConv:3}.

\paragraph{Step 2} 
Applying Proposition~\ref{p:MartConv} with~$f$ as above and~$h\equiv\nabla f$, the process~$B_\bullet\eqdef\seq{B_t}_{t\in [0,T]}$ with
\[
B_t\eqdef \mu_t f - \mu_0 f-\int_0^t \mu_s\tparen{\nabla f\cdot H(\mu_s)} \diff s \comma \qquad t\in [0,T]\comma
\]
is well-defined and a continuous martingale w.r.t.~$\msF_\bullet$ with quadratic variation
\[
\quadvar{B}_t= \int_0^t \mu_s\abs{\nabla f}^2 \diff s = c\, t \comma \qquad t\in [0,T] \fstop
\]

By L\'evy's characterization, the process~$W_\bullet\eqdef\seq{W_t}_{t\in [0,T]}$ with~$W_t\eqdef B_{t/c}$, is a standard one-dimensional Brownian motion.
Note that~$c\eqdef \mu_0(\mbbM^d)>0$ is $\msF_0$-measurable, therefore it is independent of~$W_\bullet$ since the latter is an $\msF_\bullet$-Brownian motion, see e.g.~\cite[Prob.~2.5.5, p.~73]{KarShr88}.
As a consequence, the set
\[
E\eqdef \set{B_T<-c \norm{H}_0 T}
\]
has positive $\mbfP$-probability.

On the one hand, on the set of positive probability~$E$,
\begin{align*}
\mu_T f &= \int_0^T \mu_s \tparen{\nabla f\cdot H(\mu_s)} \diff s + B_T
\\
&< \norm{H}_0 \int_0^T \mu_s\car \diff s - c \norm{H}_0 T =0 \fstop
\end{align*}
On the other hand,~$\mu_T f$ is a.s.\ non-negative, since~$f$ is a non-negative function by the choice of~$g$ and Lemma~\ref{l:AuxFunction}\ref{i:l:AuxFunction:1}.
Thus we have reached a contradiction, as desired.
\end{proof}

\section{Possible extensions}
Let us collect here some observations about possible extensions of our main result.

Solutions to the free Dean--Kawasaki equation have been constructed in~\cite{LzDS17+,LzDS24,KonLehvRe19} in a far more general setting than~$\mbbM^d$, encompassing e.g.\ Riemannian manifolds, as well as some `non-smooth spaces'.
For the sake of simplicity, let us discuss the case of a Riemannian manifold~$M$ with Riemannian metric~$\mssg$.
A definition of solution to~\eqref{eq:DK} is given again in terms of the martingale problem in Definition~\ref{d:Solution}, replacing the Laplacian on~$\mbbM^d$ with the Laplace--Beltrami operator~$\Delta_\mssg$ on~$M$, the gradient with the Riemannian gradient~$\nabla^\mssg$ induced by~$\mssg$, and the scalar product with the metric~$\mssg$ itself.

We expect the non-existence result in Theorem~\ref{t:Main} to be a \emph{structural property} of the equation, rather than a feature of the ambient space, and thus to extend to this more general setting as well.
Indeed, given a solution~$\mu_\bullet$ up to time~$T$, the proof depends only on the construction of a function~$f\colon M\to \R^+$ satisfying~$\abs{\nabla f}\equiv \car$ on some Borel set~$A\subset M$ $\mu_t$-negligible for $\Leb^1$-a.e.~$t\in [0,T]$.
To control this negligibility when~$M=\mbbM^d$, we introduced the measure~$\mu^*_T$ in~\eqref{eq:AuxMeas} as the time average of the marginal of~$\mu_\bullet$ on~$\mbbM^1$ with respect to the projection onto the \emph{first coordinate}.
On a general manifold, this can be done by choosing~$\mu^*_T$ as the time average of the marginal of~$\mu_\bullet$ on~$\R^+_0$ with respect to the projection onto the \emph{radial coordinate} in a spherical coordinate system centered at any point~$o\in M$, viz.
\[
\mu^*_T[0,r) \eqdef \int_0^T \mu_s B_r(o) \diff s\comma \qquad T>0\comma \quad r>0\comma
\]
where~$B_r(o)$ is the ball in~$M$ of radius~$r>0$ and center~$o$ w.r.t.\ the intrinsic distance~$\mssd_\mssg$ on~$M$ induced by~$\mssg$.
A function~$g\colon \R^+_0\to\R^+$ may then be constructed from Lemma~\ref{l:AuxFunction}, so that~$f(x)=g\tparen{\mssd_\mssg(x,o)}$ has the desired properties.

For a general manifold~$M$, the argumentation above is not sufficient to prove the conclusion, since we also need to show that~$\mu_s$ vanishes on the singular set~$\Sigma_f\subset M$ of~$f$, and this set includes the \emph{cut locus} of the point~$o$, which is generally `large' and wildly dependent on~$o$.
However, the argument can be made rigorous on manifolds with only \emph{one chart}, (including Euclidean spaces, hyperbolic spaces, etc.) in which the cut locus of any~$o$ is empty, and on standard spheres, in which the cut locus of~$o$ exactly consists of its antipodal point.


\begin{thebibliography}{10}

\bibitem{AdaDirPelZim11}
S.~Adams, N.~Dirr, M.~A. Peletier, and J.~Zimmer.
\newblock {From a Large-Deviations Principle to the Wasserstein Gradient Flow:
  A New Micro-Macro Passage}.
\newblock {\em {Commun.\ Math.\ Phys.}}, 307(3):791--815, 2011.
\newblock \href {https://doi.org/10.1007/s00220-011-1328-4}
  {\path{doi:10.1007/s00220-011-1328-4}}.

\bibitem{BenKipLan95}
O.~Benois, C.~Kipnis, and C.~Landim.
\newblock {Large deviations from the hydrodynamical limit of mean zero
  asymmetric zero range processes}.
\newblock {\em {Stoch.\ Proc.\ Appl.}}, 55(1):65--89, January 1995.
\newblock \href {https://doi.org/10.1016/0304-4149(95)91543-a}
  {\path{doi:10.1016/0304-4149(95)91543-a}}.

\bibitem{BerDeSGabJon15}
L.~Bertini, A.~De~Sole, D.~Gabrielli, G.~Jona-Lasinio, and C.~Landim.
\newblock Macroscopic fluctuation theory.
\newblock {\em {Rev.\ Mod.\ Phys.}}, 87(2):593--636, June 2015.
\newblock \href {https://doi.org/10.1103/revmodphys.87.593}
  {\path{doi:10.1103/revmodphys.87.593}}.

\bibitem{CheEng05}
A.~S. Cherny and H.-J. Engelbert.
\newblock {\em {Singular Stochastic Differential Equations}}, volume 1858 of
  {\em {Lecture Notes in Mathematics}}.
\newblock Springer Berlin Heidelberg, 2005.
\newblock \href {https://doi.org/10.1007/b104187} {\path{doi:10.1007/b104187}}.

\bibitem{CorFis23b}
F.~Cornalba and J.~Fischer.
\newblock {Multilevel Monte Carlo methods for the Dean--Kawasaki equation from
  Fluctuating Hydrodynamics}.
\newblock {\em {arXiv:2311.08872}}, 2023.
\newblock \href {https://doi.org/10.48550/ARXIV.2311.08872}
  {\path{doi:10.48550/ARXIV.2311.08872}}.

\bibitem{CorFis23}
F.~Cornalba and J.~Fischer.
\newblock {The Dean--Kawasaki Equation and the Structure of Density
  Fluctuations in Systems of Diffusing Particles}.
\newblock {\em {Arch.\ Rational Mech.\ Anal.}}, 247(5), August 2023.
\newblock \href {https://doi.org/10.1007/s00205-023-01903-7}
  {\path{doi:10.1007/s00205-023-01903-7}}.

\bibitem{CorFisIngRai23}
F.~Cornalba, J.~Fischer, J.~Ingmanns, and C.~Raithel.
\newblock {Density fluctuations in weakly interacting particle systems via the
  Dean--Kawasaki equation}.
\newblock {\em {arXiv:2303.00429}}, 2023.
\newblock \href {https://doi.org/10.48550/ARXIV.2303.00429}
  {\path{doi:10.48550/ARXIV.2303.00429}}.

\bibitem{Dea96}
D.~S. Dean.
\newblock {Langevin equation for the density of a system of interacting
  Langevin processes}.
\newblock {\em {J.~Phys.~A: Math.\ Gen.}}, 29:L613--L617, 1996.

\bibitem{LzDS17+}
L.~Dello~Schiavo.
\newblock {The Dirichlet--Ferguson Diffusion on the Space of Probability
  Measures over a Closed Riemannian Manifold}.
\newblock {\em {Ann.\ Probab.}}, 50(2):591--648, 2022.
\newblock \href {https://doi.org/10.1214/21-AOP1541}
  {\path{doi:10.1214/21-AOP1541}}.

\bibitem{LzDS24}
L.~Dello~Schiavo.
\newblock {Massive Particle Systems, Wasserstein Brownian Motions, and the
  Dean--Kawasaki Equation}.
\newblock {\em {arXiv:2411.14936}}, pages 1--103, 2024.
\newblock \href {https://doi.org/10.48550/arXiv.2411.14936}
  {\path{doi:10.48550/arXiv.2411.14936}}.

\bibitem{DirStaZim16}
N.~Dirr, M.~Stamatakis, and J.~Zimmer.
\newblock {Entropic and gradient flow formulations for nonlinear diffusion}.
\newblock {\em {J.~Math.\ Phys.}}, 57(8), August 2016.
\newblock \href {https://doi.org/10.1063/1.4960748}
  {\path{doi:10.1063/1.4960748}}.

\bibitem{DjuKrePer24}
A.~Djurdjevac, H.~Kremp, and N.~Perkowski.
\newblock {Weak error analysis for a nonlinear SPDE approximation of the
  Dean--Kawasaki equation}.
\newblock {\em {Stoch.\ PDE: Anal.\ Comp.}}, March 2024.
\newblock \href {https://doi.org/10.1007/s40072-024-00324-1}
  {\path{doi:10.1007/s40072-024-00324-1}}.

\bibitem{EmbDirZimRei18}
P.~Embacher, N.~Dirr, J.~Zimmer, and C.~Reina.
\newblock {Computing diffusivities from particle models out of equilibrium}.
\newblock {\em {Proc.\ R.~Soc.~A}}, 474(2212):20170694, April 2018.
\newblock \href {https://doi.org/10.1098/rspa.2017.0694}
  {\path{doi:10.1098/rspa.2017.0694}}.

\bibitem{FehGes23}
B.~Fehrman and B.~Gess.
\newblock {Non-equilibrium large deviations and parabolic-hyperbolic PDE with
  irregular drift}.
\newblock {\em {Invent.\ math.}}, 234(2):573--636, July 2023.
\newblock \href {https://doi.org/10.1007/s00222-023-01207-3}
  {\path{doi:10.1007/s00222-023-01207-3}}.

\bibitem{FehGes24b}
B.~Fehrman and B.~Gess.
\newblock {Conservative stochastic PDEs on the whole space}.
\newblock {\em {arXiv:2410.00254}}, 2024.
\newblock \href {https://doi.org/10.48550/ARXIV.2410.00254}
  {\path{doi:10.48550/ARXIV.2410.00254}}.

\bibitem{FehGes24}
B.~Fehrman and B.~Gess.
\newblock {Well-Posedness of the Dean--Kawasaki and the Nonlinear
  Dawson--Watanabe Equation with Correlated Noise}.
\newblock {\em {Arch.\ Rational Mech.\ Anal.}}, 248(2), March 2024.
\newblock \href {https://doi.org/10.1007/s00205-024-01963-3}
  {\path{doi:10.1007/s00205-024-01963-3}}.

\bibitem{GesGvaKon22}
B.~Gess, R.~S. Gvalani, and V.~Konarovskyi.
\newblock {Conservative SPDEs as fluctuating mean field limits of stochastic
  gradient descent}.
\newblock {\em {Probab.\ Theor.\ Relat. Fields (to appear)}}, 2022.
\newblock \href {https://doi.org/10.48550/ARXIV.2207.05705}
  {\path{doi:10.48550/ARXIV.2207.05705}}.

\bibitem{JacZim14}
R.~L. Jack and J.~Zimmer.
\newblock {Geometrical interpretation of fluctuating hydrodynamics in diffusive
  systems}.
\newblock {\em {J.~Phys.~A: Math.\ Theor.}}, 47(48):485001, November 2014.
\newblock \href {https://doi.org/10.1088/1751-8113/47/48/485001}
  {\path{doi:10.1088/1751-8113/47/48/485001}}.

\bibitem{KarShr88}
I.~Karatzas and S.~E. Shreve.
\newblock {\em {Brownian Motion and Stochastic Calculus}}.
\newblock Springer US, 1988.
\newblock \href {https://doi.org/10.1007/978-1-4684-0302-2}
  {\path{doi:10.1007/978-1-4684-0302-2}}.

\bibitem{Kaw94}
K.~Kawasaki.
\newblock {Stochastic model of slow dynamics in supercooled liquids and dense
  colloidal suspensions}.
\newblock {\em {Physica A: Statist.\ Mech.\ Appl.}}, 208(1):35--64, July 1994.
\newblock \href {https://doi.org/10.1016/0378-4371(94)90533-9}
  {\path{doi:10.1016/0378-4371(94)90533-9}}.

\bibitem{KonLehvRe19}
V.~Konarovskyi, T.~Lehmann, and M.-K. von Renesse.
\newblock {Dean-Kawasaki dynamics: ill-posedness vs.~triviality}.
\newblock {\em {Electron.\ Commun.\ Probab.}}, 24, jan 2019.
\newblock \href {https://doi.org/10.1214/19-ecp208}
  {\path{doi:10.1214/19-ecp208}}.

\bibitem{KonLehvRe19b}
V.~Konarovskyi, T.~Lehmann, and M.-K. von Renesse.
\newblock {On Dean-Kawasaki Dynamics with Smooth Drift Potential}.
\newblock {\em {J.~Statist.\ Phys.}}, 178(3):666--681, nov 2019.
\newblock \href {https://doi.org/10.1007/s10955-019-02449-3}
  {\path{doi:10.1007/s10955-019-02449-3}}.

\bibitem{KonMue23}
V.~Konarovskyi and F.~M{\"{u}}ller.
\newblock {Dean--Kawasaki equation with initial condition in the space of
  positive distributions}.
\newblock {\em {J.~Evol.\ Equ.}}, 24(4), October 2024.
\newblock \href {https://doi.org/10.1007/s00028-024-01018-w}
  {\path{doi:10.1007/s00028-024-01018-w}}.

\bibitem{KonvRe18}
V.~Konarovskyi and M.-K. von Renesse.
\newblock {Modified Massive Arratia flow and Wasserstein diffusion}.
\newblock {\em {Comm.\ Pure Appl.\ Math.}}, 72(4):764--800, 2019.
\newblock \href {https://doi.org/10.1002/cpa.21758}
  {\path{doi:10.1002/cpa.21758}}.

\bibitem{KonvRe17}
V.~Konarovskyi and M.-K. von Renesse.
\newblock {Reversible coalescing-fragmentating Wasserstein dynamics on the real
  line}.
\newblock {\em {J.~Funct.\ Anal.}}, 286(8):110342, April 2024.
\newblock \href {https://doi.org/10.1016/j.jfa.2024.110342}
  {\path{doi:10.1016/j.jfa.2024.110342}}.

\bibitem{MuevReZim24}
F.~M{\"{u}}ller, M.-K. von Renesse, and J.~Zimmer.
\newblock {Well-Posedness for Dean--Kawasaki Models of Vlasov-Fokker-Planck
  Type}.
\newblock {\em {arXiv:2411.14334}}, nov 2024.
\newblock \href {https://doi.org/10.48550/arXiv.2411.14334}
  {\path{doi:10.48550/arXiv.2411.14334}}.

\bibitem{QuaRezVar99}
J.~Quastel, F.~Rezakhanlou, and S.~R.~S. Varadhan.
\newblock {Large deviations for the symmetric simple exclusion process in
  dimensions $d\geq3$}.
\newblock {\em {Probab.\ Theory Relat.\ Fields}}, 113(1):1--84, February 1999.
\newblock \href {https://doi.org/10.1007/s004400050202}
  {\path{doi:10.1007/s004400050202}}.

\bibitem{VruLoeWitt20}
M.~te~Vrugt, H.~L{\"{o}}wen, and R.~Wittkowski.
\newblock Classical dynamical density functional theory: from fundamentals to
  applications.
\newblock {\em {Adv.\ Phys.}}, 69(2):121--247, April 2020.
\newblock \href {https://doi.org/10.1080/00018732.2020.1854965}
  {\path{doi:10.1080/00018732.2020.1854965}}.

\end{thebibliography}

\end{document}